\documentclass[12pt,a4paper]{amsart}
\usepackage{latexsym,amsfonts,amsthm,amsmath,mathrsfs,amssymb}
\usepackage[british]{babel}
\usepackage[dvips]{color}
\usepackage[utf8]{inputenc}
\usepackage[T1]{fontenc}
\usepackage[dvips,final]{graphics}
\usepackage{xcolor}
\usepackage{mathrsfs}
\usepackage{mathtools} 
\usepackage{breqn}
\usepackage{epstopdf}
\usepackage{verbatim}
\usepackage{amscd}
\usepackage{hyperref}
\usepackage[english,capitalize]{cleveref}
\usepackage{aliascnt}

\usepackage{todonotes}

\makeatletter
\def\newaliasedtheorem#1[#2]#3{
	\newaliascnt{#1@alt}{#2}
	\newtheorem{#1}[#1@alt]{#3}
	\expandafter\newcommand\csname #1@altname\endcsname{#3}
}
\makeatother

\theoremstyle{plain}

\newtheorem{theorem}{Theorem}[section]
\newaliasedtheorem{prop}[theorem]{Proposition}
\newaliasedtheorem{lem}[theorem]{Lemma}
\newaliasedtheorem{coroll}[theorem]{Corollary}

\theoremstyle{definition}
\newaliasedtheorem{defi}[theorem]{Definition}
\newaliasedtheorem{quest}[theorem]{Question}
\newaliasedtheorem{fact}[theorem]{Fact}

\theoremstyle{remark}
\newaliasedtheorem{rem}[theorem]{Remark}
\newaliasedtheorem{exa}[theorem]{Example}

\newcommand{\R}{\mathbb{R}}
\newcommand{\N}{\mathbb{N}}

\let\altphi\phi
\let\phi\varphi
\let\varphi\altphi
\let\altphi\undefined

\newcommand{\average}{{\mathchoice {\kern1ex\vcenter{\hrule height.4pt
width 6pt
depth0pt} \kern-9.7pt} {\kern1ex\vcenter{\hrule height.4pt width 4.3pt
depth0pt}
\kern-7pt} {} {} }}

\address{\textsc{Daniela Di Donato}: 
Dipartimento di Ingegneria Industriale e Scienze Matematiche, Via Brecce Bianche, 12 60131 Ancona, Universit\'a Politecnica delle Marche.}
\email{d.didonato@staff.univpm.it}

\title{The intrinsic Hopf-Lax semigroup vs. The intrinsic slope}

\date{\today}

\author{ Daniela Di Donato }

\begin{document}

\begin{abstract}

		In this note, we introduce a natural notion of intrinsic Hopf-Lax semigroup in the context of the so-called intrinsically Lipschitz sections.  The main aims  are to  prove the  link between the intrinsic Hopf-Lax semigroup and  the intrinsic slope and to show that the intrinsic Hopf-Lax semigroup is a subsolution of Hamilton-Jacobi type equality. 
		\end{abstract}

\maketitle 
\tableofcontents

\section{Introduction} 	
Let $M$ be a compact Riemannian manifold, then the quadratic Hamilton-Jacobi equation on $M$ is
  \begin{equation*}
\frac{\partial F}{\partial t} + \frac{|\nabla F|^2}{2}=0.
\end{equation*}
Given an initial condition $f\in C(M, \R)$, it is well-known  the viscosity solution to the Hamilton-Jacobi equation is given by the Hopf-Lax formula
\begin{equation*}
F(y,t)= \inf_{z\in M} \left\{ f(z) + \frac{d^2(y,z)}{2t} \right\},
\end{equation*}
where $d$ is the geodesic distance on $M$ and $t\in \R^+$. The map that sends $f$ to $F(\cdot, t)$ defines a semigroup action of $\R^+$ on $C(M,\R)$, called the Hamilton-Jacobi semigroup or Hopf-Lax semigroup. Motivated by the fact that these semigroups are the key ingredients in \cite{AGS08.2, LV07}, we define a natural notion of the "intrinsic Hopf-Lax semigroup" (see Definition \ref{defiHopf-Lax semigroup}) in the intrinsic context introduced in \cite{DDLD21}. More precisely,  in general metric spaces, Le Donne and the author give a 'different' notion of Lipschitz graph starting from two simple facts:
\begin{enumerate}
\item we generalize the notion of intrinsically Lipschitz maps in subRiemannian Carnot groups \cite{ABB, BLU, CDPT}  introduced and studied by Franchi, Serapioni and Serra Cassano  \cite{FSSC, FSSC03, MR2032504}.
\item we consider graphs instead of maps.
\end{enumerate}

 In our context we consider a section $\phi :Y \to X$ of  $\pi:X \to Y$ (i.e., $\pi \circ \phi = id_Y$) such that $\pi$ produces a foliation for $X,$ i.e., $X= \coprod \pi ^{-1} (y)$ and the Lipschitz property of $\phi$ consists to ask that the distance between two points $\phi (y_1), \phi (y_2)$ is not comparable with the distance between $y_1$ and $y_2$ but between $\phi (y_1)$ and the fiber of $y_2$.  Following this idea, it is natural to define the intrinsic Hopf-Lax semigroup as follows. Let $X=\R^\kappa$, $Y\subset \R^\kappa$ be bounded and $\pi :X \to Y$ a quotient map. 
The intrinsic Hopf-Lax semigroup is the family of operators $(iQ_\cdot)_{t>0}$ defined as 
\begin{equation*}\label{equation231}
f \,\, \mapsto \,\, iQ_tf(y) := \inf_{z\in Y} \left\{ \max_{j=1,\dots , \kappa} f_j(z) +\frac 1 {2t} d^2(f(z), \pi^{-1} (y)) \right\}.
\end{equation*}
for any continuous section $f:Y \to X$ of $\pi.$ We also consider the case $\kappa >1$ because being intrinsically Lipschitz is equivalent to Lipschitz in the classical sense  when we consider the basic case $X=Y=\R.$

Yet, in the context of metric measure spaces, there are different notions of "energies" like Cheeger energy and Dirichlet form and a natural question is  to ask when a Dirichlet form is regular (i.e., when it coincides with the Cheeger energy). The first step of this study is to consider the Hopf-Lax semigroup and to prove its link with  the descending slope. The reader can see \cite{C99,  ACDM15, AES16, AGS08.1,  AGS08.2, AGS08.3, AGS08.4, BGL01, KSZ14, KZ12, FOT10, KM16, K04}.


 
Here we prove the following results:
\begin{enumerate}
\item Theorem \ref{theoremHopfLaxFormulanew}: we estimate of the time derivative of the Hopf-Lax semigroup in terms of $D^\pm f$ (see Definition  \ref{defiHopf-Lax semigroup}). 
\item Theorem \ref{theoremHopfLaxDUALITY}: we prove  the  link between the intrinsic Hopf-Lax semigroup and  the intrinsic slope.
\item Corollary \ref{provaHJno}: we show that the intrinsic Hopf-Lax semigroup is a subsolution of Hamilton-Jacobi type equality.
\end{enumerate}
The long-term objective is to obtain \textbf{the regularity of a suitable Dirichlet form in our intrinsic context.} This question arises from the fact that  the key point to obtain the regularity of a classical Dirichlet form turns out to be the existence of a "suitable" Lipschitz approximation for any function $f:Y \to \R$ inside the domain of the form. Yet, in our context, it is plausible that the appropriate approximation will be in terms of the intrinsically Lipschitz sections because they play the same role of classic Lipschitz notion in many results as proved in \cite{DDLD21, D22.31may, D22.1, D22.24june, D22.29}.

 {\bf Acknowledgements.}  We would like to thank Professor Giuseppe Savar\'e for the reference  \cite{AGS08.2} which is the core of this paper.

\section{Intrinsic   Lipschitz  sections}
\subsection{Intrinsic   Lipschitz  sections} The notion of intrinsically Lipschitz maps was introduced by Franchi, Serapioni and Serra Cassano  \cite{FSSC, FSSC03, MR2032504}  (see also  \cite{SC16, FS16}) in the context of subRiemannian Carnot groups after a negative result given by Ambrosio and  Kirchheim \cite{AmbrosioKirchheimRect}. Their aim is to establish a good definition of rectifiability in Carnot groups.

Here we present a generalization of this concept introduced in \cite{DDLD21}. Our setting is the following. We have a metric space $X$, a topological space $Y$, and a 
quotient map $\pi:X\to Y$, meaning
continuous, open, and surjective.
The standard example for us is when $X$ is a metric Lie group $G$ (meaning that the Lie group $G$ is equipped with a left-invariant distance that induces the manifold topology), for example a subRiemannian Carnot group, 
and $Y$ if the space of left cosets $G/H$, where 
$H<G$ is a  closed subgroup and $\pi:G\to G/H$ is the projection modulo $H$, $g\mapsto gH$.
\begin{defi}\label{def_ILS} 
We say that a map $\phi :Y \to X$ is a section of $\pi $ if 
\begin{equation}\label{equation1}
\pi \circ \phi =\mbox{id}_Y.
\end{equation}
Moreover, we say that a map $\phi:Y\to X$ is an {\em intrinsically Lipschitz section of $\pi$ with constant $L$},  with $L\in[1,\infty)$, if in addition
\begin{equation}\label{equationFINITA}
d(\phi (y_1), \phi (y_2)) \leq L d(\phi (y_1), \pi ^{-1} (y_2)), \quad \mbox{for all } y_1, y_2 \in Y.
\end{equation}
Here $d$ denotes the distance on $X$, and, as usual, for a subset $A\subset X$ and a point $x\in X$, we have
$d(x,A):=\inf\{d(x,a):a\in A\}$.
\end{defi}
A first observation is that this class is contained in the class of continuous maps (see \cite[Section 2.4]{DDLD21}) but  cannot be uniformly continuous (see Example 1.2 in \cite{D22.24june}). 
 In the case  $ \pi$ is a Lipschitz quotient or submetry \cite{MR1736929, Berestovski},  being intrinsically Lipschitz  is equivalent to biLipschitz embedding, see Proposition 2.4 in \cite{DDLD21}. Moreover, since $\phi$ is injective by \eqref{equation1}, the class of Lipschitz sections not include the constant maps.

   \begin{rem}\label{costuniversale}
   If $Y$ is bounded, we get that \begin{equation*}
K:= \sup _{y_1,y_2 \in Y}  d(\phi (y_1), \pi ^{-1} (y_2)) <+ \infty .
\end{equation*}
This follows because, on the contrary, if $K=+ \infty,$ then we get the contradiction $+ \infty= d(\phi (y_1), \pi ^{-1} (y_2))  \leq d(\phi (y_1), \phi (y_2)).$
  \end{rem}

\subsection{Intrinsic   Lipschitz  constants}
 We recall the definition of the intrinsic Lipschitz constants as in \cite{D22.31may, D22.29}, where we have adapted the theory of  \cite{C99, DM} in our intrinsic case.

\begin{defi}\label{def_ILS.1} Let $\phi:Y\to X$ be a section of $\pi$. Then we define 
\begin{equation*}
ILS (\phi):= \sup _{\substack{y_1, y_2 \in Y \\ y_1\ne y_2}} \frac{d(\phi (y_1), \phi (y_2))}{  d(\phi (y_1), \pi ^{-1} (y_2)) } \in [0, + \infty ]
\end{equation*}
and
\begin{equation*}
\begin{aligned}
ILS (Y,X,\pi ) &:= \{ \phi :Y \to X \,:\, \phi \mbox{ is an intrinsically Lipschitz section of $\pi$ and }  ILS(\phi) < +\infty \},\\
ILS_{b} (Y,X,\pi) & := \{ \phi \in  ILS (Y,X,\pi) \,:\, \mbox{spt}(\phi) \mbox{ is bounded} \}.\\
\end{aligned}
\end{equation*}
\textbf{For simplicity, we will write $ILS (Y,\R ^\kappa)$ instead of  $ILS (Y,\R ^\kappa ,\pi ).$}
\end{defi}

\begin{defi}\label{def_ILS.2} Let $\phi:Y\to X$ be a  section of $\pi$. Then we define the local intrinsically Lipschitz constant (also called intrinsic slope) of $\phi$ the map $Ils (\phi):Y \to [0,+\infty )$ defined as 
\begin{equation*}
Ils (\phi) (z):= \limsup _{y\to z} \frac{d(\phi (y), \phi (z))}{  d(\phi (y), \pi ^{-1} (z)) },
\end{equation*}
if $z \in Y$ is an accumulation point; and $Ils (\phi) (z):=0$ otherwise.
\end{defi}

\begin{defi}\label{def_ILS.3} Let $\phi:Y\to X$ be a section of $\pi$. Then we define the asymptotic  intrinsically Lipschitz constant of $\phi$ the map $Ils_a (\phi):Y \to [0,+\infty )$ given by
\begin{equation*}
Ils_a (\phi) (z):= \limsup _{y_1,y_2\to z}\frac{d(\phi (y_1),\phi (y_2))}{  d(\phi (y_1), \pi ^{-1} (y_2)) }
\end{equation*}
if $z \in Y$ is an accumulation point and $Ils (\phi) (z):=0$ otherwise.
\end{defi}
 
 \begin{rem}\label{defrem} Notice that by $\phi (y_2) \in \pi ^{-1} (y_2),$ it is trivial that $ d(\phi (y_1), \pi ^{-1} (y_2)) \leq d(\phi (y_1), \phi (y_2))$ and so $Ils(\phi) \geq 1.$ Moreover, it holds 
\begin{equation*}
 Ils (\phi ) \leq Ils_a (\phi ) \leq ILS(\phi).
\end{equation*}
\end{rem}

%

\section{The intrinsic Hopf-Lax semigroup}
\subsection{The intrinsic Hopf-Lax semigroup: Definition}
 In this section, we give a natural definition of the Hopf-Lax semigroup in our intrinsic context: the classical one is widely used in different situations, from metric measure spaces theory to optimal transportation.   We used the similar technique shown in \cite[Section 3]{AGS08.2} (see also \cite{LV07, P17}). 
 
 For any fixed $t\in \R^+$ we give the key definition of this note.
 \begin{defi}\label{defiHopf-Lax semigroup}  Let $X=\R^\kappa$, $Y\subset \R^\kappa$ be bounded and $\pi :X \to Y$ a quotient map. 
The intrinsic Hopf-Lax semigroup is the family of operators $(iQ_\cdot)_{t>0}$ defined as 
\begin{equation}\label{equation231}
f \,\, \mapsto \,\, iQ_tf(y) := \inf_{z\in Y}F(t,y,z),
\end{equation}
for any section $f=(f_1,\dots , f_\kappa ) \in C(Y, \R^\kappa)$ of $\pi,$ where
\begin{equation}\label{equation9luglio}
F(t,y,z)= \max_{j=1,\dots , \kappa} f_j(z) +\frac 1 {2t} d^2(f(z), \pi^{-1} (y)) .
\end{equation}
Given a continuous section $f:Y \to \R^\kappa $ of $\pi$, $iQ_tf(y)$ is then defined by the minimum problem \eqref{equation231}.
We also define:
\begin{equation}
\begin{aligned}
iD^+ f(y,t) & := \sup \left\{ \limsup_{n\to \infty}  d(f(y_n), \pi^{-1} (y)) \, :\, (y_n)_n \mbox{ is a minimizing sequence in } \eqref{equation231} \right\}\\
iD^- f(y,t) & := \inf \left\{ \liminf_{n\to \infty}  d(f(y_n), \pi^{-1} (y)) \, :\, (y_n)_n \mbox{ is a minimizing sequence in  }\eqref{equation231} \right\}\\
\end{aligned}
\end{equation}
\end{defi}

The map $(y,t) \mapsto iQ_tf(y)$, $Y\times (0,\infty) \to \R \cup \{ \pm \infty\}$ is obviously upper semicontinuous. The behavior of $iQ_tf$ is not trivial only in the set
\begin{equation*}
\{ y\in Y \,:\,  d(f(z), \pi^{-1} (y)) <\infty \mbox{ for some $z\in Y$ with } \max_j f_j(z)<\infty \},
\end{equation*}
and so we shall restrict our analysis in this set. In particular, it is sufficient to ask that $f$ is a bounded section and $Y$ is a bounded subset of $\R^\kappa$ (see Remark \ref{costuniversale}). Moreover, it this set, $iQ_tf (y)\in \R \cup \{ -\infty\}$ and so we also define
\begin{equation*}
t_*(y) := \sup \{ t\in \R^+\,:\, iQ_tf(y) >-\infty \},
\end{equation*}
 with the convention $t_*(y)=0$ if $iQ_tf(y) =-\infty $ for all $t>0.$

 \begin{rem}
Notice that by Remark \ref{costuniversale}, if $Y$ is bounded, then $iD^+ f <+\infty .$ Moreover, in general we have that
\begin{equation*}
iD^+ f (y,t) \geq iD^- f(y,t), \quad \forall y\in Y, t\in \R^+.
\end{equation*} 
\end{rem}

 \begin{prop}[Semicontinuity of $iD^\pm$]\label{propSEMICONT} Let $y_n\to y$ and $t_n \to t \in (0,t_*(y)).$ Then,
 \begin{equation*}
\begin{aligned}
iD^- f(y,t) &\leq \liminf_{n\to \infty}  iD^- f(y_n,t_n), \\
iD^+ f(y,t) &\geq \limsup_{n\to \infty}  iD^+ f(y_n,t_n). \\
\end{aligned}
\end{equation*}
In particular, for every $y\in Y$ the map $t\mapsto iD^-f(y,t)$ is left continuous in $(0,t_*(y))$ and the map  $t\mapsto iD^+f(y,t)$ is right continuous in $(0, t_*(y)).$
\end{prop}

  \begin{proof} 
  We adapt the proof as in \cite[Proposition 3.2]{AGS08.2}. For every $n\in \N,$ let $(y_n^\ell)_{\ell \in \N}$ be a minimizing sequence for $F(t_n, y_n,\cdot )$ for which the limit of $d(f(y_n^\ell), \pi^{-1} (y_n)) $ as $\ell \to \infty $ equals $iD^- (y_n,t_n).$ By Remark \ref{costuniversale}, $\sup _{\ell, n}  d(f(y_n^\ell), \pi^{-1} (y_n)) <+\infty $ and for any $n$ we have that
  \begin{equation*}
\lim_{\ell \to \infty} f(y^\ell_n) +\frac 1 {2t_n}   d^2(f(y_n^\ell), \pi^{-1} (y_n)) = iQ_ {t_n} f(y_n).
\end{equation*}
Moreover, the upper semicontinuity of $(y,t) \mapsto iQ_tf(y)$ gives that $\limsup _n iQ_{t_n} f(x_n) \leq iQ_tf(y).$ Since $ d(f(y_n^\ell), \pi^{-1} (y_n)) $ is bounded, it follows that $$\sup_\ell | d^2(f(y_n^\ell), \pi^{-1} (y_n))  -  d^2(f(y_n^\ell), \pi^{-1} (y)) | $$ is infinitesimal and so by a diagonal argument we can find a sequence $n\mapsto \ell (n) $ such that
 \begin{equation*}
\begin{aligned}
& \limsup_{ n\to \infty} f(y^{\ell(n)})  + \frac 1 {2t}  d^2(f(y_n^\ell), \pi^{-1} (y)) \leq iQ_tf(y),\\
& | d(f(y_n^{\ell(n)}), \pi^{-1} (y_n))  -  iD^-(y_n, t_n) |  \leq \frac 1 n .\\
\end{aligned}
\end{equation*}
This implies that $y\mapsto y^{\ell (n)}_n $ is a minimizing sequence for $F(t,y,\cdot ),$ therefore 
\begin{equation*}
D^- f(y,t) \leq \liminf_{ n\to \infty} d(f(y_n^{\ell(n)}), \pi^{-1} (y)) = \liminf_{ n\to \infty} d(f(y_n^{\ell(n)}), \pi^{-1} (y_n))  = \liminf_{ n\to \infty} D^-f(y_n, t_n).
\end{equation*}
Notice that in the equality we used that fact that $y_n \to y$ which we will prove in Proposition \ref{theoremHopfLax} (ii).  
In a similar way, if we choose instead sequence $(y^\ell_n)_\ell$ on which the supremum in the definition of $iD^+f(y_n, t_n)$ is attained, we obtain the upper semicontinuity property of $iD^+f$.

  \end{proof}

  We conclude this section with an easy result when $f$ is an intrinsically Lipschitz section.
  
 \begin{prop}\label{theoremHopfLax.2}
Let $f \in ILS(Y, \R ^\kappa)$ and let $L \geq 1$ be its Lipschitz constant. Then, 
$$2t L \geq iD^+f(y,t) \geq iD^-f(y,t),$$ for every $y\in Y$ and $t\in \R^+.$

\end{prop}

\begin{proof}
We can suppose $iQ_t f(y)<\max_{j=1,\dots , \kappa} f_j(y);$ indeed, if not, it must be $$iQ_tf(y) =\max_{j=1,\dots , \kappa} f_j(y) \quad  \Rightarrow \quad iD^+ f(y,t) =0.$$ Hence we can take a minimizing sequence $(y_n)_n$ for $iQ_tf(y)$ so that definitively 
\begin{equation*}
\max_{j=1,\dots , \kappa} f_j(y_n) +\frac 1 {2t} d^2(f(y_n), \pi^{-1} (y)) \leq \max_{j=1,\dots , \kappa} f_j(y).
\end{equation*}
Using the fact that $f$ in an intrinsically $L$-Lipschitz sections, it follows that
 \begin{equation*}
d^2(f(y_n), \pi^{-1} (y))  \leq 2t d(f(y), f(y_n)) \leq 2t L d(f(y_n), \pi^{-1} (y)).
\end{equation*}
Dividing for $d(f(y_n), \pi^{-1} (y))$ and taking the limsup in $n$, we find the thesis.
 \end{proof}

\subsection{The intrinsic Hopf-Lax semigroup: Properties}
The next theorem treats the main properties of the intrinsic Hopf-Lax semigroup. Like in the classical case, note that no completeness is needed and there is no reference measure. In particular, the point $(iv)$ of the next result shows that $iD^\pm f(y,\cdot )$ are non increasing and they coincide out of a countable set (see Remark \ref{remnondecreasing}).

 \begin{prop}\label{theoremHopfLax}
Let $f \in C_b(Y,\R ^\kappa)$ be a section of  $\pi :\R^\kappa \to Y$. Then we have the following basic properties for $iQf(y):$
\begin{description}
\item[ i]  $\inf _{z\in Y} \min_{j=1,\dots , \kappa} f_j \leq iQ_tf \leq \max_{j=1,\dots , \kappa} f_j \leq \sup _{z\in Y} \max_{j=1,\dots , \kappa} f_j< + \infty.$
\item[ ii]  $iQ_tf \to \max_{j=1,\dots , \kappa} f_j$ pointwise as $t\to 0.$
\item[ iii] for any fixed $y\in Y,$ the map $t\mapsto iQ_tf(y)$ is Lipschitz in the classical sense in $(\delta , t_*(y))$ for all $\delta $ such that $0<\delta <t_*(y)$ and the Lipschitz constant depends on $\delta$ and $Osc(f):= \sup \max_{j=1,\dots , \kappa} f_j -\inf \max_{j=1,\dots , \kappa} f_j .$
\item[ iv]  for every $y\in Y$ and $0<t<s < t_*(y),$ it holds: $iD^+ f (y,t) \leq iD^- f(y,s).$
\end{description}

\end{prop}

  \begin{proof} We adapt the proof as in \cite[Theorem 2.3.3]{P17}.
  
$(i).$ It is a trivial consequence of the fact that $\frac 1 {2t} d^2(f(z), \pi^{-1} (z)) \geq 0$ and that we can use $z=y$ as a competitor in the infimum of  \eqref{equation231}.

$(ii).$ Fix $y\in Y$ and take a sequence $t_n \to 0$; consider a quasi-minimizing sequence $(y_n)_n$ for $iQ_{t_n}f(y)$, in the sense that:
\begin{equation*}
iQ_{t_n} f(y) +\frac 1 n \geq \max_{j=1,\dots , \kappa} f_j(y_n) + \frac 1 {2t_n} d^2(f(y_n) , \pi^{-1} (y)), \quad \forall n\in \N .
\end{equation*}
Firstly, note that the uniform bound given by $(i)$ and $f\in C_b(Y, \R^\kappa)$ yields
\begin{equation}\label{equation24giugno}
f(y_n) \to f(y).
\end{equation}
Indeed,
\begin{equation*}
d^2 (f(y_n) , \pi^{-1} (y)) \leq 2 t_n \left(iQ_{t_n} f(y) +\frac 1 n - \max_{j=1,\dots , \kappa} f_j(y_n) \right) \leq 2t_n \left(2\|f \|_\infty +\frac 1 n \right) \,\, \longrightarrow 0,
\end{equation*}
and so 
\begin{equation*}
f(y_n) \to a \in \pi^{-1} (y).
\end{equation*}
Moreover, notice that $\pi(f(y_n))=y_n$ and using the fact that $f$ is a section, we get $y_n \to \pi (a)=y.$ Consequently, \eqref{equation24giugno} holds by continuity of $f$. Now, since the simply fact
\begin{equation*}
iQ_{t_n} f(y) +\frac 1 n \geq  f_\ell(y_n) + \frac 1 {2t_n} d^2(f(y_n) , \pi^{-1} (y)), \quad \forall n\in \N ,
\end{equation*} for every $\ell =1,\dots , \kappa ,$ if $\max_{j=1,\dots , \kappa} f_j(y)=f_o(y),$ then
\begin{equation*}
f_o(y) \geq \lim_{n\to +\infty } iQ_{t_n} f(y) \geq \lim_{n\to +\infty} f_o(y_n) = f_o(y),
\end{equation*}
and so we get the second point $(ii).$

$(iii).$ Fix $y\in Y, \varepsilon >0$ and consider $s,t \in\R$ with $s>t.$ Let's take an $\varepsilon $-quasi minimum $y_s$ for $iQ_sf(y),$ i.e.,
\begin{equation*}
iQ_{s} f(y) +\varepsilon \geq \max_{j=1,\dots , \kappa} f_j(y_s) + \frac 1 {2s} d^2(f(y_s) , \pi^{-1} (y)).
\end{equation*}
 Now using $y_s$ as a competitor for the inf-problem of $iQ_tf(y)$ we get
 \begin{equation}\label{equation236}
\begin{aligned}
& |iQ_t f(y)-iQ_sf(y)| -\varepsilon  = iQ_t f(y)-iQ_sf(y) -\varepsilon \\
&\,\, \leq \max_{j=1,\dots , \kappa} f_j(y_s) + \frac 1 {2t} d^2(f(y_s) , \pi^{-1} (y)) - \max_{j=1,\dots , \kappa} f_j(y_s) - \frac 1 {2s} d^2(f(y_s) , \pi^{-1} (y)) -\varepsilon \\
&\,\, \leq \frac{|s-t|}{2ts} d^2(f(y_s), \pi^{-1} (y)) -\varepsilon \\
\end{aligned}
\end{equation}
We need to have some control on the distance along the quasi minimizing sequence $(y_s)_s$; to get it, note that we can confine our "attention" in the inf-problem of $i Q_sf(y)$ to a subset of $Y$ (dependent on $s,y$); precisely we can suppose to work with $z$ inside the set
\begin{equation}\label{equation237}
\left\{ z\in Y \, : \, \frac 1 {2s} d^2 (f(z), \pi^{-1} (y)) \leq \sup \max_{j=1,\dots , \kappa} f_j -\inf \max_{j=1,\dots , \kappa} f_j =: Osc (f)\right\}.
\end{equation}
Indeed, if we take $z\in Y$ such that does not satisfy the inequality in \eqref{equation237}, we deduce that
\begin{equation*} \begin{aligned}
\max_{j=1,\dots , \kappa} f_j(z) +\frac 1 {2s} d^2(f(z), \pi^{-1} (y)) & > \inf \max_{j=1,\dots , \kappa} f_j +\sup \max_{j=1,\dots , \kappa} f_j - \inf \max_{j=1,\dots , \kappa} f_j \\
& =\sup \max_{j=1,\dots , \kappa} f_j \geq \max_{j=1,\dots , \kappa} f_j(z) \geq iQ_s f(y).
\end{aligned}
\end{equation*}
Hence, without loss of generality, we can suppose:
\begin{equation*}
y_s \in B(y, \sqrt{2s\, Osc (f)}) \subset Y.
\end{equation*}
Thanks to \eqref{equation236}, it holds
  \begin{equation*}
\begin{aligned}
|iQ_t f(y)-iQ_sf(y)| & \leq \frac{|s-t|}{2ts} d^2(f(y_s), \pi^{-1} (y)) + \varepsilon \\
& \leq \frac{|s-t|}{t} Osc (f) + \varepsilon  \\
&\leq \frac{|s-t|}{\delta } Osc (f) + \varepsilon , 
\end{aligned}
\end{equation*}
for all $\varepsilon >0$ and recall that $t\in (\delta , +\infty).$ By the arbitrariness of $\varepsilon,$ this gives us the sought uniformly Lipschitzianity with respect to $t,$ as desired.

$(iv).$  Fix $0<t<s$ and $y\in Y.$ Let's make this proof under the additional condition that the infimum in $iQ_tf(y)$ and in $iQ_sf(y)$ are both attained and so they are minima (if not one should arrange a bit the proof but it is mainly the same idea). Hence take $y_t, y_s$ minima related respectively to $iQ_tf(y)$ and $iQ_sf(y)$  and, by definition of the intrinsic Hopf Lax semigroup, we deduce that
  \begin{equation*}
\begin{aligned}
\max_{j=1,\dots , \kappa} f_j(y_t) +\frac 1 {2t} d^2(f(y_t), \pi^{-1} (y)) & \leq  \max_{j=1,\dots , \kappa} f_j(y_s) +\frac 1 {2t} d^2(f(y_s), \pi^{-1} (y)),\\
  \max_{j=1,\dots , \kappa} f_j (y_s) +\frac 1 {2s} d^2(f(y_s), \pi^{-1} (y)) & \leq  \max_{j=1,\dots , \kappa} f_j(y_t) +\frac 1 {2s} d^2(f(y_t), \pi^{-1} (y)).\\
\end{aligned}
\end{equation*}
Summing up the previous equations, it holds
  \begin{equation*}
\begin{aligned}
\left(\frac 1 {2t} - \frac 1 {2s} \right) d^2(f(y_t), \pi^{-1} (y)) & \leq \left(\frac 1 {2t} - \frac 1 {2s} \right)  d^2(f(y_s), \pi^{-1} (y)),\\
\end{aligned}
\end{equation*}
and so, recall that $s>t$ and $1/s < 1/t,$ we obtain
\begin{equation*}
d^2(f(y_t), \pi^{-1} (y)) \leq d^2(f(y_s), \pi^{-1} (y)).
\end{equation*}
Now let the square root in the last inequality, $d(f(y_t), \pi^{-1} (y)) \leq d(f(y_s), \pi^{-1} (y))$ holds. More precisely, $$d(f(y_t), \pi^{-1} (y)) \leq d(f(y_s), \pi^{-1} (y))$$ is true for every choice $(y_s, y_t)$ into the class of minimizers of $iQ_sf(y)$ and $iQ_tf(y),$ respectively. This gives us the sought inequality and the proof of the statement is complete.
 \end{proof}

   \begin{rem}
We notice that in Theorem $\ref{theoremHopfLax}$ (iii) we do not expect the Lipschitzianity in the second variable. We do not know if $iQ_tf$ is intrinsic Lipschitz in the second variable but this proof does not work because in our context we consider the distance between a set and a point instead of between two points.  

\end{rem}
 
   \begin{rem}\label{remnondecreasing}
   It holds
   \begin{equation}\label{equationKEY25}
 iD^+f(y,t) = iD^-f(y,t), \quad \mbox{for a.e. } t\in \R^+.
\end{equation}
Indeed,
\begin{equation*}
\sup_{s<t} iD^+f(y,s) \leq iD^-f(y,t) \leq  \inf_{s>t} iD^+f(y,s),
\end{equation*}
for every $y\in Y$ and $t\in \R^+.$ Notice that in the first inequality we used Theorem \ref{theoremHopfLax} (v). Moreover, given $y\in Y,$ we have that
\begin{equation*}
\sup_{s<t} iD^+f(y,s) = iD^+f(y,t) = \inf_{s>t} iD^+f(y,s),
\end{equation*}
for any $t\in \R^+,$ where $iD^+f(y, \cdot)$ is continuous; nevertheless, $iD^+f(y, \cdot)$ is non decreasing and so there are at most countable many points of discontinuity. Consequently, \eqref{equationKEY25} is true. 

   \end{rem}
 As corollaries of Theorem $\ref{theoremHopfLax}$ we get the following results.
   \begin{coroll}\label{coroll1}
Under the assumption of Theorem $\ref{theoremHopfLax}$,  for every fixed $y\in Y$ the map
\begin{equation*}
t\,\, \mapsto \,\, iQ_t f(y)
\end{equation*}
turns out to be differentiable almost everywhere with respect to the $L^1$-measure in $\R^+$.
\end{coroll}

 \begin{proof}
It is enough to consider Theorem $\ref{theoremHopfLax}$ (iii) and Rademacher's Theorem.
 \end{proof}


 \subsection{The time derivative of $iQ_tf$}
We find a precise estimate of the time derivative of the Hopf-Lax semigroup in terms of $D^\pm f(y,t);$ in order to do this fact, we give an alternative proof of Lipschitz property of $iQ_tf.$ Moreover, we recall that semiconcave map $g$ on an open interval are local quadratic perturbations of concave maps; they inherit from concave functions all pointwise differentiability properties, as existence of right and left derivatives $\frac{d^-}{ dt} g \geq \frac{d^+}{ dt} g$ which is important for the next result.
  \begin{prop}[Time derivative of $iQ_tf$]\label{theoremHopfLaxFormulanew}
The map $(0,t_*(y)) \ni t \mapsto iQ_t f(y)$ is locally Lipschitz and locally semiconcave. For all $t\in (0,t_*(y))$ it satisfies  
\begin{equation}\label{equationIMPO}
\begin{aligned}
\frac{d^-}{ dt} iQ_t f(y) & =- \frac{ (iD^-f(y,t) )^2}{ 2t^2},\\
\frac{d^+}{ dt} iQ_t f(y) & =- \frac{ (iD^+f(y,t) )^2}{ 2t^2},\\
\end{aligned}
\end{equation} In particular, $t\mapsto iQ_tf(y)$ is differentiable at $t\in (0, t_*(y))$ if and only if $iD^+f(y,t)= iD^-f(y,t).$
\end{prop}

   \begin{proof} We follows \cite[Proposition 3.3]{AGS08.2}. Let $(y_t^n)_n, (y_s^n)_n$ be minimizing sequences for $F(t,y, \cdot)$ and $F(s,y, \cdot )$. We have
   \begin{equation}\label{equationIMPO.81}
\begin{aligned}
 iQ_s f(y) - iQ_t f(y) & \leq \liminf_{ n \to \infty } F(s,y, y_t^n) - F(t,y, y_t^n )\\
 &  = \liminf_{ n \to \infty } \frac{ d^2(f(y_t^n)), \pi^{-1} (y))}{2}\left( \frac 1 s - \frac 1 t \right),
\end{aligned}
\end{equation} 
\begin{equation}\label{equationIMPO.82}
\begin{aligned}
 iQ_s f(y) - iQ_t f(y) & \geq \limsup_{ n \to \infty } F(s,y, y_s^n) - F(t,y, y_s^n )\\
 &  = \limsup_{ n \to \infty } \frac{ d^2(f(y_s^n)), \pi^{-1} (y))}{2}\left( \frac 1 s - \frac 1 t \right).
\end{aligned}
\end{equation} 
Now we have two cases:  $s>t$ or vice versa. In the first case, we get
\begin{equation}\label{equationIMPO.83}
\begin{aligned}
 \frac{ (iD^- f(y,s))^2}{2}\left( \frac 1 s - \frac 1 t \right) \leq  iQ_t f(y) - iQ_s f(y) & \leq  \frac{  (iD^+ f(y,t))^2 }{2}\left( \frac 1 s - \frac 1 t \right),
\end{aligned}
\end{equation} 
 recalling that $\lim_{s\to t} iD^-f(y,s) =iD^+f(y,t),$ a division by $s-t$ (noting that $1/s-1/t=-(s-t)/st$) and a limit as $s\to t $ gives the identity for the right derivative in \eqref{equationIMPO}. In a similar way, we can obtain the left derivative.
 
 Moreover, the local Lipschitz continuity follows by \eqref{equationIMPO.83} recalling that $iD^\pm f(y,\cdot )$ are locally bounded functions; we easily get the quantitative bound 
 \begin{equation*}
\left\|\frac d {dt} iQ_tf(y) \right\|_{L^\infty (\tau_1, \tau _2)} \leq \frac 1 {2\tau_1^2} \| iD^+f(y,\cdot)\|_{L^\infty (\tau_1, \tau _2)},
\end{equation*}
for every $0<\tau_1 < \tau_2 < t_*(y).$ Finally, since the distributional derivative of the map $t\mapsto (iD^+f(y,t))^2 /(2t^2)$ is locally bounded from below, we also deduce that $t\mapsto iQ_tf$ is locally semiconcave, as desired. Hence, the proof is complete.
 \end{proof}
 
   \begin{rem}\label{coroll9l}
We want to underline that since $Y\subset \R^\kappa$ is bounded, then the map $(0,t_*(y)) \ni t \mapsto iQ_t f(y)$ is $globally$ Lipschitz. 

  This fact follows from \eqref{equationIMPO.83}, noticing that $iD^\pm f(y,\cdot )$ are globally bounded functions by Remark \ref{costuniversale}.
\end{rem}

 \begin{prop}\label{theoremHJstep1}
Let $f:Y \to \R^\kappa$ be a continuous section of $\pi.$ Then, it holds
  \begin{equation}\label{equation3.13ab}
\begin{aligned}
& t\in (0, t_*(y)) \quad \Rightarrow \quad \limsup_{y\to z } \frac{iQ_tf(z)- iQ_tf(y) }{d(f(y), \pi^{-1} (z))}\leq \frac{iD^+f(z,t)} {t}\\
\end{aligned}
\end{equation}
\end{prop}

 \begin{proof}
We use a similar technique as in \cite[Proposition 3.4]{AGS08.2}.

Let $y,z \in Y$ such that $iQ_tf(y) >-\infty.$ We want to show that
\begin{equation}\label{equation3.14}
iQ_tf(z)- iQ_tf(y) \leq d(f(z), \pi^{-1} (y)) \left( \frac{iD^-f(y,t)} t + \frac{ d(f(z), \pi^{-1} (y)) } {2t}\right).
\end{equation}
Let $(y_n)_n$ be a minimizing sequence for $F(t,y,\cdot )$ on which the infimum in the definition of $iD^-f(y,t)$ is attained, obtaining 
   \begin{equation*}
\begin{aligned}
iQ_tf(z)- iQ_tf(y) & \leq \liminf_{ n\to \infty} F(t,z,y_n) - F(t,y,y_n)\\
& =\liminf _{ n\to \infty} \frac{d^2(f(y_n), \pi^{-1} (z))} {2t}  -  \frac{d^2(f(y_n), \pi^{-1} (y))} {2t}\\
& \leq \liminf _{ n\to \infty} \frac{d(f(y), \pi^{-1} (z))} {2t} ( d(f(y_n), \pi^{-1} (z))  + d(f(y_n), \pi^{-1} (y))  ) \\
& \leq \frac{d(f(y), \pi^{-1} (z))} {2t} ( d(f(y), \pi^{-1} (z)) + iD^-f(y,t)),
\end{aligned}
\end{equation*}
where in the second inequality we used the fact that $f(y_n) \to f(y)$ (see \eqref{equation24giugno}). Hence \eqref{equation3.14} holds. Now dividing both sides of \eqref{equation3.14} by $d(f(y), \pi^{-1} (z))$ and taking the $\limsup$ as $y\to z$ we get the first inequality of \eqref{equation3.13ab}, since Proposition \ref{propSEMICONT} yields the upper-semicontinuity of $iD^+f.$ 
\end{proof}

  \subsection{$iQ_tf$ as a subsolution of Hamilton-Jacobi type inequality}
In our case, we don't know if $iQ_tf$ is a subsolution of Hamilton-Jacobi inequality as in \cite[Theorem 3.5]{AGS08.2}; however, we have the following corollary.
 \begin{coroll}\label{provaHJno}
Let $f:Y \to \R^\kappa$ be a continuous section of $\pi$ with $Y\subset \R^\kappa$ bounded. Then, it holds
  \begin{equation*}
\begin{aligned}
\frac{d^+}{ dt} iQ_t f(y) +\frac 1 2  \limsup_{y\to z }\left( \frac{iQ_tf(z)- iQ_tf(y) }{d(f(y), \pi^{-1} (z))} \right)^2 \leq 0.
\end{aligned}
\end{equation*}
\end{coroll}

 \begin{proof}
It is enough to consider Proposition \ref{theoremHopfLaxFormulanew} and \ref{theoremHJstep1}.
 \end{proof}

  \section{Intrinsic Hopf-Lax semigroup vs. Intrinsic slope}
  
 Our main goal now is to prove a "duality formula" for the intrinsic slope of a continuous section using the intrinsic Hopf-Lax semigroup. In the classical case, this result is the first step in order to get the regularity of Dirichlet form. Here, we adapt the proof of \cite[Theorem 2.3.6]{P17}.

   \begin{theorem}\label{theoremHopfLaxDUALITY}
Let $f:Y\to \R^\kappa $ be a continuous section of $\pi.$ Then,
 \begin{equation*}
Ils(f)^2 (y) \geq 2 \limsup_{t\to 0} \frac{ \max_{j=1,\dots , \kappa} f_j (y) -iQ_tf(y)}{t}.
\end{equation*}

\end{theorem}

 \begin{proof}
Fix $y\in Y$ and let $\ell \in \{1,\dots , \kappa \}$ such that $\max_{j=1,\dots , \kappa} f_j (y_t)= f_\ell (y_t)$ We consider a quasi-minimizing sequence $(y_t)_{t>0}$ for $iQ_tf(y)$, i.e.,
 \begin{equation*}
  \begin{aligned}
iQ_t f(y) +\varepsilon _t &  \geq  f_\ell (y_t) +\frac{1}{2t} d^2(f(y_t), \pi ^{-1} (y)),
\end{aligned}
\end{equation*}
where, without loss of generality,
\begin{equation*}
\lim_{t\to 0} \frac{\varepsilon _t}{t}=0.
\end{equation*}
Yet, without loss of generality, we can suppose that $iQ_tf(y) <  f_\ell(y)$ definitively in $t\to 0.$ Hence, noting
\begin{equation*}
 f_\ell (y)- iQ_t f(y) -\varepsilon _t \leq   f_\ell(y) - f_\ell(y_t) -\frac{1}{2t} d^2(f(y_t), \pi ^{-1} (y))\leq  d(f(y), f(y_t)) -\frac{1}{2t} d^2(f(y_t), \pi ^{-1} (y))
\end{equation*}
we have that
  \begin{equation*}
\begin{aligned}
& \limsup _{t\to 0}  \frac{ f_\ell(y) - iQ_tf(y)}{t}  - \frac{\varepsilon _t}{ t} \\
& \leq \limsup _{t\to 0} \frac{ d^2( f(y),  f(y_t))}{d(f(y_t), \pi ^{-1} (y))}  \frac{d(f(y_t), \pi ^{-1} (y))}{t}  - \frac{d^2(f(y_t), \pi ^{-1} (y))}{2t^2} \\
& \leq  \limsup _{t\to 0}\frac 1 2  \frac{d^2( f(y),  f(y_t))}{d^2(f(y_t), \pi ^{-1} (y))} +\frac 1 2  \frac{d^2(f(y_t), \pi ^{-1} (y))}{t^2}  - \frac{d^2(f(y_t), \pi ^{-1} (y))}{2t^2}, \\
& = \frac 1 2 Ils (f)^2  (y),
\end{aligned}
\end{equation*}
where in the last inequality we used Young inequality and in the last equality we used the fact that $y_t\to y$ as $t\to 0$ (see \eqref{equation24giugno}). 
  \end{proof}

 \bibliographystyle{alpha}
\bibliography{DDLD}

\newcommand{\etalchar}[1]{$^{#1}$}
\begin{thebibliography}{ACDM15}

\bibitem[ABB19]{ABB}
Andrei Agrachev, Davide Barilari, and Ugo Boscain.
\newblock A comprehensive introduction to sub-{R}iemannian geometry.
\newblock {\em Cambridge Studies in Advanced Mathematics, Cambridge Univ.
  Press}, 181:762, 2019.

\bibitem[ACDM15]{ACDM15}
L.~Ambrosio, M.~Colombo, and S.~Di~Marino.
\newblock Sobolev spaces in metric measure spaces: reflexivity and lower
  semicontinuity of slope.
\newblock {\em Adv. Stud. Pure Math.}, 67:1--58, 2015.

\bibitem[AES16]{AES16}
L.~Ambrosio, M.~Erbar, and G.~Savar\'{e}.
\newblock Optimal transport, {C}heeger energies and contractivity of dynamic
  transport distances in extended spaces.
\newblock {\em Nonlinear Analysis: Theory, Methods and Applications},
  137:77--134, 2016.

\bibitem[AGS08]{AGS08.1}
L.~Ambrosio, N.~Gigli, and G.~Savar\'e.
\newblock Gradient flows in metric spaces and in the space of probability
  measures.
\newblock {\em Lectures in Mathematics ETH Z\"urich, Birkh\"auser Verlag,
  Basel, second edition}, 2008.

\bibitem[AGS14a]{AGS08.2}
L.~Ambrosio, N.~Gigli, and G.~Savar\'e.
\newblock Calculus and heat flow in metric measure spaces and applications to
  spaces with {R}icci bounds from below.
\newblock {\em Invent. Math.}, 195(2):289--391, 2014.

\bibitem[AGS14b]{AGS08.3}
L.~Ambrosio, N.~Gigli, and G.~Savar\'e.
\newblock Metric measure spaces with {R}iemannian {R}icci curvature bounded
  from below.
\newblock {\em Duke Math. J.}, 163(7):1405--1490, 2014.

\bibitem[AGS15]{AGS08.4}
L.~Ambrosio, N.~Gigli, and G.~Savar\'e.
\newblock Bakry-\'{E}mery curvature-dimension condition and {R}iemannian
  {R}icci curvature bounds.
\newblock {\em Ann. Probab.}, 43(1):339--404, 2015.

\bibitem[AK00]{AmbrosioKirchheimRect}
Luigi Ambrosio and Bernd Kirchheim.
\newblock Rectifiable sets in metric and {B}anach spaces.
\newblock {\em Math. Ann.}, 318(3):527--555, 2000.

\bibitem[BGL01]{BGL01}
S~Bobkov, I~Gentil, and M.~Ledoux.
\newblock Hypercontractivity of {H}amilton-{J}acobi equations.
\newblock {\em J. Math. Pures Appl.}, 80:669--696, 2001.

\bibitem[BJL{\etalchar{+}}99]{MR1736929}
S.~Bates, W.~B. Johnson, J.~Lindenstrauss, D.~Preiss, and G.~Schechtman.
\newblock Affine approximation of {L}ipschitz functions and nonlinear
  quotients.
\newblock {\em Geom. Funct. Anal.}, 9(6):1092--1127, 1999.

\bibitem[BLU07]{BLU}
A.~Bonfiglioli, E.~Lanconelli, and F.~Uguzzoni.
\newblock {\em Stratified {L}ie groups and potential theory for their
  sub-{L}aplacians}.
\newblock Springer Monographs in Mathematics. Springer, Berlin, 2007.

\bibitem[CDPT07]{CDPT}
Luca Capogna, Donatella Danielli, Scott~D. Pauls, and Jeremy~T. Tyson.
\newblock {\em An introduction to the {H}eisenberg group and the
  sub-{R}iemannian isoperimetric problem}, volume 259 of {\em Progress in
  Mathematics}.
\newblock Birkh\"{a}user Verlag, Basel, 2007.

\bibitem[Che99]{C99}
J.~Cheeger.
\newblock Differentiability of {L}ipschitz functions on metric measure spaces.
\newblock {\em Geom. Funct. Anal.}, 9(8):428--517, 1999.

\bibitem[DD22a]{D22.31may}
Daniela Di~Donato.
\newblock Intrinsic {C}heeger energy for the intrinsically {L}ipschitz
  constants.
\newblock {\em preprint}, 2022.

\bibitem[DD22b]{D22.24june}
Daniela Di~Donato.
\newblock Intrinsic {L}ipschitz sections of no-linear quotient maps.
\newblock {\em preprint}, 2022.

\bibitem[DD22c]{D22.1}
Daniela Di~Donato.
\newblock Intrinsically {H}\"older sections in metric spaces.
\newblock {\em preprint}, 2022.

\bibitem[DD22d]{D22.29}
Daniela Di~Donato.
\newblock A note about intrinsically {L}ipschitz constants.
\newblock {\em preprint}, 2022.

\bibitem[DDLD22]{DDLD21}
Daniela Di~Donato and Enrico Le~Donne.
\newblock Intrinsically {L}ipschitz sections and applications to metric groups.
\newblock {\em preprint}, 2022.

\bibitem[DM14]{DM}
S.~Di~Marino.
\newblock Recent advances in {BV} and {S}obolev spaces in metric measure
  spaces.
\newblock {\em PhD Thesis in Mathematics}, 2014.

\bibitem[FOM10]{FOT10}
M.~Fukushima, Y.~Oshima, and Takeda M.
\newblock Dirichlet forms and symmetric {M}arkov processes.
\newblock {\em Walter de Gruyter}, 19, 2010.

\bibitem[FS16]{FS16}
Bruno Franchi and Raul~Paolo Serapioni.
\newblock Intrinsic {L}ipschitz graphs within {C}arnot groups.
\newblock {\em J. Geom. Anal.}, 26(3):1946--1994, 2016.

\bibitem[FSSC01]{FSSC}
B.~Franchi, R.~Serapioni, and F.~Serra~Cassano.
\newblock Rectifiability and perimeter in the {H}eisenberg group.
\newblock {\em Math. Ann.}, 321(3):479--531, 2001.

\bibitem[FSSC03a]{MR2032504}
B.~Franchi, R.~Serapioni, and F.~Serra~Cassano.
\newblock Regular hypersurfaces, intrinsic perimeter and implicit function
  theorem in {C}arnot groups.
\newblock {\em Comm. Anal. Geom.}, 11(5):909--944, 2003.

\bibitem[FSSC03b]{FSSC03}
Bruno Franchi, Raul Serapioni, and Francesco Serra~Cassano.
\newblock On the structure of finite perimeter sets in step 2 {C}arnot groups.
\newblock {\em The Journal of Geometric Analysis}, 13(3):421--466, 2003.

\bibitem[Kei04]{K04}
S.~Keith.
\newblock A differentiable structure for metric measure spaces.
\newblock {\em Adv. Math. 183}, pages 271--315, 2004.

\bibitem[KM16]{KM16}
B.~Kleiner and J.M. Mackay.
\newblock Differentiable structures on metric measure spaces: a primer.
\newblock {\em Ann. Sc. Norm. Super. Pisa Cl. Sci. (5) Vol. XVI}, pages 41--64,
  2016.

\bibitem[KSY14]{KSZ14}
P.~Koskela, N.~Shanmugalingam, and Zhou Y.
\newblock Geometry and analysis of {D}irichlet forms (ii).
\newblock {\em Journal of Functional Analysis}, 267:2437--2477, 2014.

\bibitem[KY12]{KZ12}
P.~Koskela and Zhou Y.
\newblock Geometry and analysis of {D}irichlet forms.
\newblock {\em Advances in Mathematics}, 231:2755--2801, 2012.

\bibitem[LV07]{LV07}
J.~Lott and C.~Villani.
\newblock Hamilton-{J}acobi semigroup on length spaces and applications.
\newblock {\em Journal de math\'ematiques pures et appliqu\'ees}, 88:219--229,
  2007.

\bibitem[Por17]{P17}
L.~Portinale.
\newblock Metric measure spaces and {U}pper regularity of {D}irichlet forms.
\newblock {\em PhD Thesis in Mathematics, Pisa}, 2017.

\bibitem[SC16]{SC16}
Francesco Serra~Cassano.
\newblock Some topics of geometric measure theory in {C}arnot groups.
\newblock In {\em Geometry, analysis and dynamics on sub-{R}iemannian
  manifolds. {V}ol. 1}, EMS Ser. Lect. Math., pages 1--121. Eur. Math. Soc.,
  Z\"urich, 2016.

\bibitem[VN88]{Berestovski}
Berestovskii Valerii~Nikolaevich.
\newblock Homogeneous manifolds with intrinsic metric.
\newblock {\em Sib Math J}, I(29):887--897, 1988.

\end{thebibliography}

\end{document}